\documentclass[leqno,12pt]{article} 
\title{Separable Integer Partition Classes} 
\author{by\\George E. Andrews}
\date{}

\usepackage{dsfont}
\usepackage{amsmath,amsthm}
\usepackage{tikz}
\allowdisplaybreaks

\usepackage[all]{xy}
\xyoption{arc}

\usepackage{amsfonts}
\usepackage{enumerate}

\numberwithin{equation}{section}

\newtheorem{theorem}{Theorem}
\newtheorem*{theorem*}{Theorem}
\newtheorem{corollary}[theorem]{Corollary}

\newtheorem*{1GGtheorem}{First G\"ollnitz-Gordon Theorem}
\newtheorem*{Stheorem}{Schur's Theorem}
\newtheorem*{RStheorem}{Refinement of Schur's Theorem}
\newtheorem*{Glasgow}{Glasgow Mod 8 Theorem}

\theoremstyle{remark}
\newtheorem*{remark}{Remark}

\theoremstyle{definition}

\newtheorem{lemma}[theorem]{Lemma}
\newtheorem*{lemma*}{Lemma}
\newtheorem*{conjecture*}{Conjecture}

\newcommand{\E}{\mathcal{E}}

\newcommand{\BS}{\mathcal{B}_\mathcal{S}}
\newcommand{\PS}{\mathcal{P}_\mathcal{S}}

\DeclareMathOperator{\MEX}{MEX}

\begin{document}
\maketitle



\begin{abstract}
A classical method for partition generating function is developed into a tool with wide applications. New expansions of well-known theorems are derived, and new results for partitions with $n$ copies of $n$ are presented.
\end{abstract}

\section{Introduction}\label{Intro}

The object of this paper is to systematize a process in the theory of integer partition that really dates back to Euler. It is epitomized in the partition-theoretic interpretation of three classical identities:
\begin{align}
\label{1.1} & \sum_{n\geq 0} \frac{q^n}{(q;q)_n}=\frac{1}{(q;q)_\infty},\\
\label{1.2} & \sum_{n\geq 0} \frac{q^{n(n+1)/2}}{(q;q)_n}=(-q;q)_\infty,
\end{align}
and
\begin{equation}
\label{1.3} \sum_{n\geq 0} \frac{q^{n^2}}{(q;q)_n} = \frac{1}{(q;q^5)_\infty(q^4;q^5)_\infty},
\end{equation}
where
\begin{equation}
\label{1.4} (A;q)_N=(1-A)(1-Aq)\cdots(1-Aq^{N-1}).
\end{equation}

Equations \eqref{1.1} and \eqref{1.2} are Euler's \cite[p.~19]{9} while \eqref{1.3} is the first of the celebrated Rogers-Ramanujan identities \cite[Ch.~7]{9}.

In section~\ref{section2}, we will analyze \eqref{1.1}-\eqref{1.3} from the point of view of separable integer partition classes. 

A separable integer partition class (SIP), $\mathcal{P}$, with modulus $k$, is a subset of all the integer partitions. In addition, there is a subset $\mathcal{B}\subset\mathcal{P}$ ($\mathcal{B}$ is called the basis of $\mathcal{P}$) such that for each integer $n\geq 1$, the number of elements of $\mathcal{B}$ with $n$ parts is finite and every element of $\mathcal{P}$ with $n$ parts is uniquely of the form
\begin{equation}
\label{1.5} (b_1+\pi_1)+(b_2+\pi_2)+\cdots+(b_n+\pi_n)
\end{equation}
where $0<b_1\leq b_2 \leq \ldots \leq b_n$ are a partition in $\mathcal{B}$ and $0\leq \pi_1\leq \pi_2 \leq \ldots \leq \pi_n$ is a partition into $n$ nonnegative parts, whose only restriction is that each part is divisible by $k$. Furthermore, all partitions of the form \eqref{1.6} are in $\mathcal{P}$. 

As we will see in section~\ref{section2}, each of \eqref{1.1}-\eqref{1.3} can be developed from this point of view with modulus $k=1$. However. this setting allows a similar examination of the first G\"ollnitz-Gordon identity \cite{4} in section~\ref{section3}:
\begin{equation}
\label{1.6} \sum_{n\geq 0} \frac{q^{n^2}(-q;q^2)_n}{(q^2;q^2)_n}=\frac{1}{(q;q^8)_\infty(q^4;q^8)_\infty(q^7;q^8)_\infty}.
\end{equation}
More surprising is an analysis of Schur's 1926 partition theorem \cite{20} in section~4. It is interesting to note that this analysis leads naturally to full proofs of both \eqref{1.4} and Schur's theorem.

In section~\ref{section5}, we examine a new application. In section~\ref{section6}, we extend the concept of SIP classes to partitions with $n$ copies of $n$ \cite{1}. Section~\ref{section7} overpartitions.
In the final section, we discuss open questions and note that the idea first arose in dynamical systems concerning billiard trajectories \cite{12}.

\section{General Theory and Classical Identities}\label{section2}

The infinite series in each of \eqref{1.1}-\eqref{1.3} fit neatly into the S.I.P. program with modulus $k=1$. We begin with \eqref{1.1}. In this case, we let $\mathcal{P}_\mathds{N}$ be the set of all integer partitions. Now for each $n$, there is only one element of $\mathcal{B}_\mathds{N}$ with $n$ parts, namely
$$
1+1+1+\cdots+1,
$$
and every element of $\mathcal{P}_\mathds{N}$ with $n$ parts, say $\pi_1+\pi_2+\cdots+\pi_n$ can be written
$$
(1+(\pi_1-1))+(1+(\pi_2-1))+\cdots+(1+(\pi_n-1)).
$$
The generating function for the elements of $\mathcal{P}_\mathds{N}$ with $n$ parts is therefore
$$
\frac{q^{1+1+\cdots+1}}{(1-q)(1-q^2)\ldots(1-q^n)}=\frac{q^n}{(q;q)_n}.
$$
Summing over all $n$ yields
$$
\sum_{n\geq 0} \frac{q^n}{(q;q)_n},
$$
the left-hand side of \eqref{1.1}.

Next we consider $\mathcal{P}_\mathcal{D}$, the integer partitions that have distinct parts. Here for each $n$ there is again exactly one partition in $\mathcal{B}_\mathcal{D}$ with $n$ parts, namely
$$
1+2+3+\cdots+n,
$$
and every element of $\mathcal{P}_\mathcal{D}$, say
$$
\pi_1+\pi_2+\cdots+\pi_n\qquad(0\leq \pi_1\leq \pi_2 \leq \ldots \leq \pi_n)
$$
can be written
$$
(1+\pi_1)+(2+\pi_2)+\cdots+(n+\pi_n).
$$
Furthermore
$$
\pi_1+\pi_2+\cdots+\pi_n
$$
constitutes an ordinary partition into $n$ nonnegative parts.

The generating function is therefore
$$
\sum_{n\geq 0} \frac{q^{1+2+\cdots+n}}{(q;q)_n}=\sum_{n\geq 0} \frac{q^{n(n+1)/2}}{(q;q)_n}.
$$

Finally, if $\mathcal{P}_R$ ($R$ for Rogers and Ramanujan) is the set of integer partitions where the difference between parts is $\geq 2$, the only element of $\mathcal{B}_R$ with $n$ parts is
$$
1+3+5+\cdots+(2n-1),
$$ 
and, as with \eqref{1.2}, we obtain the generating function for the partitions in $\mathcal{P}_R$ as
$$
\sum_{n\geq 0} \frac{q^{1+3+5+\cdots+(2n-1)}}{(q;q)_n}=\sum_{n\geq 0} \frac{q^{n^2}}{(q;q)_n}.
$$

The above analysis of the series in \eqref{1.1}-\eqref{1.3} is far from new. Indeed these are proofs whose ideas date back to Euler and were discussed fully in centuries old number theory and combinatorics books (cf.~\cite[Sec.~7, Ch.~III]{19}, \cite[Ch.~19]{22}) for this way of looking at \eqref{1.1}-\eqref{1.3}.

Perhaps the reason that this type of study has not gone farther is the fact that in each of the classical cases there was only one element of $\mathcal{B}$ with $n$ parts.

As we will see in the remaining sections, there are many SIP classes with a number of elements of $\mathcal{B}$ with $n$ parts. The real challenge in each instance will be to determine the generating function for $\mathcal{B}$. Obviously if we denote by $b_\mathcal{B}(n)$ the generating function for these elements of $\mathcal{B}$, then the generating function for all the partitions in $\mathcal{P}$ is given by
$$
\sum_{n\geq 0} \frac{b_\mathcal{B}(n)}{(q^k;q^k)_n},
$$
where $k$ is the modulus associated with $\mathcal{P}$. In the three cases just considered, we hardly need to think about $\mathcal{B}$ since in each case, there is only one element of $\mathcal{B}$ with $n$ parts.

So how does one determine $b_\mathcal{B}(n)$. The idea is to refine one's consideration of $b_\mathcal{B}(n,h)$ where $b_\mathcal{B}(n,h)$ is the generating function for those elements of $\mathcal{B}$ with $n$ parts and largest part $h$. Clearly
$$
b_\mathcal{B}(n)=\sum_{h\geq 0} b_\mathcal{B}(n,h).
$$

In practice we shall obtain recurrences for the $b_\mathcal{B}(n,h)$. The recurrences will arise by noting the parts in the partitions in $\mathcal{B}$ can't get too far apart. Namely if $h$ is too far from the next part then $k$ can be subtracted from $h$ yielding another partition in $\mathcal{B}$ and contradicting the uniqueness of the decomposition \eqref{1.6}. 

The previous paragraph is vague because each individual SIP class provides different meaning for ``too far from.'' 

The following theorem provides s large number of SIP classes and will facilitate the subsequent theorems in sections~\ref{section3}-\ref{section5}.

\begin{theorem}\label{theorem1} Let $\{c_1,\ldots,c_k\}$ be a set of positive integers with $c_r\equiv r$ (mod $k$) and $\{d_1,\ldots,d_k\}$ be a set of nonnegative integers. Let $\mathcal{P}$ be the set of all integers partitions
$$
b_1+b_2+\cdots+b_j
$$
where $0<b_1\leq b_2 \leq \ldots \leq b_j$, and for $1\leq r\leq k$ and each $b_i$ if $b_i \equiv r$ (mod $k$), then $b_i\geq c_r$, and if $i>1$, $b_i-b_{i-1}\geq d_r$.

Then $\mathcal{P}$ is an SIP class with modulus $k$, and $\mathcal{B}$ consists of all those partitions
$$
\beta_1+\beta_2+\cdots+\beta_j
$$
where if $\beta_1\equiv r$ (mod $k$) then $\beta_1=c_r$, and for $q\leq i\leq j$, if $\beta_i\equiv r$ (mod $k$), then 
$$
d_r\leq \beta_i-\beta_{i-1}<d_r+k.
$$
\end{theorem}

\begin{proof} We proceed by induction on the number of parts $N$ in the partition of $\mathcal{B}$. 

Clearly, if $N=1$, then we see that the single part partitions in $\mathcal{B}$ are $\{c_1,c_2,\ldots,c_k\}$. Furthermore if $b_1$ is a one part partition in $\mathcal{P}$ with $b_1\equiv c_r$ (mod $k$), then
$$
b_1=kq+c_r
$$
with $q\geq 0$.

Now suppose that our theorem holds for all partitions with fewer than $N$ parts. Let us consider an arbitrary partition $\pi$ in $\mathcal{P}$ with $N$ parts
$$
\pi_1+\pi_2+\cdots+\pi_N.
$$
From the definition of $\mathcal{P}$, we see that 
$$
\pi_1+\pi_2+\cdots+\pi_{N-1}=(\beta_1+q_1k)+(\beta_2+q_2k)+\cdots+(\beta_{N-1}+q_{N-1}k)
$$
where $\beta_1+\beta_2+\cdots+\beta_{N-1}$ is in $\mathcal{B}$ and $(q_1k)+(q_2k)+\cdots+(q_{N-1}k)$ is a partition whose parts are multiples of $k$ and $0\leq q_1\leq q_2\leq \ldots \leq q_{N-1}$. 

Now we know from the definition of $\mathcal{P}$ that if $\pi_N\equiv c_r \equiv r$ (mod $k$), then
\begin{align}
\notag \pi_N-\pi_{N-1} & = \pi_N-(\beta_{N-1}+q_{N-1}k)\\
\label{2.1} & \geq d_r.
\end{align}
Now define $\beta_N$ to be the unique integer congruent to $r$ modulo $k$ in the interval
$$
\beta_{N-1}+d_r\leq \beta_N \leq \beta_{N-1}+d_r+k.
$$
Clearly $\beta_1+\cdots+\beta_N$ is in $\mathcal{B}$. It remains to show that $q_N\geq 0$ exists so that
\begin{equation}
\label{2.2} \pi_N=\beta_N+q_Nk.
\end{equation}
Since $\pi_N\equiv r\equiv \beta_N$ (mod $k$), we need only show that $q_N$ in \eqref{2.2} is $\geq q_{N-1}$. 

Now
\begin{align*}
\beta_N+q_Nk & = \pi_N\\
& \geq d_r+\beta_{N-1}+q_{N-1}k\\
& > \beta_N-k+q_{N-1}k,
\end{align*}
thus
$$
q_Nk>(q_{N-1}-1)k,
$$
or
$$
q_N>q_{N-1}-1,
$$
i.e.
$$
q_N\geq q_{N-1}.
$$
Thus we have completed the induction step and the theorem follows.
\end{proof}

\begin{corollary}\label{corollary2} As before, $b_\mathcal{B}(n)$ denotes the generating function for partitions in $\mathcal{B}$ with exactly $n$ parts, and $P_\mathcal{P}(q)$ denotes the generating function for the partition in $\mathcal{P}$, where $\mathcal{P}$ is an SIP class of modulus $k$. Then
$$
P_\mathcal{P}(q)=\sum_{n\geq 0} \frac{b_\mathcal{B}(n)}{(q^k;q^k)_n}.
$$
\end{corollary}

\begin{proof} It is clear from the theorem that 
$$
\frac{b_\mathcal{B}(n)}{(q^k;q^k)_n}
$$
is the generating function for all partitions in $\mathcal{P}$ with exactly $n$ parts. Summing over all $n$ proves the corollary.
\end{proof}

\section{Gollnitz-Gordon}\label{section3}

Identity \eqref{1.4} is the perfect prototype to reveal how SIPs truly generalize the classical series that appear in \eqref{1.1}-\eqref{1.3}. 

First let us give the well-known partition-theoretic interpretation of \eqref{1.3} \cite{16}, \cite{17}, \cite{19}, \cite{20}:

\begin{1GGtheorem} The number of partitions of $n$ in which the difference between parts is at least 2 and at least 4 between even parts equals the number of partitions of $n$ into parts congruent to 1, 4 or 7 modulo 8.
\end{1GGtheorem}

Let $\mathcal{P}_\mathcal{G}$ denote the set of all partitions in which the difference between parts is at least 2 and at least 4 between multiples of 2.

\begin{lemma} $\mathcal{P}_\mathcal{G}$ is an SIP class of modulus 2.\end{lemma}

\begin{proof} $\mathcal{P}_\mathcal{G}$ is an instance of Theorem~\ref{theorem1} with $c_1=1$, $c_2=2$, $d_1=2$, $d_2=3$.
\end{proof}

\begin{lemma} Let $b_\mathcal{G}(n,h)$ be the generating function for the partitions in $\mathcal{B}_\mathcal{G}$ with $n$ parts and largest part equal to $n$. Then
\begin{equation}
\label{3.1} b_\mathcal{G}(1,h)=\begin{cases}
q & \text{if } h=1\\
q^2 & \text{if } h=2\\
0 & \text{otherwise}
\end{cases},
\end{equation}
and for $n>1$, $h>0$,
\begin{equation}
\label{3.2} b_\mathcal{G}(n,2h)=qb_\mathcal{G}(n,2h-1),
\end{equation}
and
\begin{equation}
\label{3.3} b_\mathcal{G}(n,2n+2h-1)=q^{n^2+h^2+2h}\genfrac[]{0pt}{0}{n-1}{h}_2,
\end{equation}
where
\begin{equation}
\label{3.4} \genfrac[]{0pt}{0}{A}{B}_n=\frac{(q^n;q^n)_A}{(q^n;q^n)_B(q^n;q^n)_{A-B}},
\end{equation}
and
\begin{equation}
\label{3.5} (A;q)_n=\prod_{j=0}^{n-1} (1-Aq^j).
\end{equation}
\end{lemma}

\begin{proof} We refer to Theorem~\ref{theorem1} to determine recurrences for $b_\mathcal{G}(n,h)$. Clearly \eqref{3.1} is immediate. 

By the conditions requiring closeness of parts as stated in Theorem~\ref{theorem1}, we see that for $n>1$,
\begin{equation}
\label{3.6} b_\mathcal{G}(n,h)=\begin{cases}
q^h(b_\mathcal{G}(n-1,h-2)+b_\mathcal{G}(n-1,h-3)) & \text{if }h\text{ is odd}\\
q^h(b_\mathcal{G}(n-1,h-3)+b_\mathcal{G}(n-1,h-4)) & \text{if }n\text{ is even.}
\end{cases}
\end{equation}

This clearly implies 
\begin{equation}
\label{3.7} b_\mathcal{G}(n,2h)=qb_\mathcal{G}(n,2h-1),
\end{equation}
which establishes \eqref{3.2}.

As for \eqref{3.3}, we see by \eqref{3.6} that
\begin{align}
\label{3.8} & b_\mathcal{G}(n,2n+2h-1)\\
\notag= & q^{2n+2h-1}\bigl(b_\mathcal{G}(n-1,2n+2h-3)+b_\mathcal{G}(n-1,2n+2h-4)\bigr)\\
\notag= & q^{2n+2h-1}\bigl(b_\mathcal{G}(n-1,2n+2h-3)+qb_\mathcal{G}(n-1,2n+2h-5)\bigr)\quad\text{(by \eqref{3.7}).}
\end{align}
Now the standard recurrence for the $q$-binomial coefficients (defined in \eqref{3.4}) namely \cite[p.~35]{9}
\begin{equation}
\label{3.8a} \genfrac[]{0pt}{0}{A}{B}_n=\genfrac[]{0pt}{0}{A-1}{B-1}_n+q^{nB}\genfrac[]{0pt}{0}{A-1}{B}_n,
\end{equation}
establishes that the right-hand side of \eqref{3.3} also satisfies the recurrence \eqref{3.8}. In addition the right-hand side of \eqref{3.3} fulfills \eqref{3.1} when $n=1$. Thus \eqref{3.3} follows by a straight forward mathematical induction on $n$.
\end{proof}

\begin{theorem}{First Gollnitz-Gordon identity}\label{theorem5} 
\begin{align}
\label{3.9} P_\mathcal{G}(q) & = \sum_{n\geq 0} \frac{(-q;q^2)_nq^{n^2}}{(q^2;q^2)_n}\\
\label{3.10} & = \frac{1}{(q;q^8)_\infty(q^4;q^8)_\infty(q^7;q^8)_\infty}.
\end{align}
\end{theorem}

\begin{remark} This celebrated theorem and its history are presented in \cite{4} and \cite[sec.~7.4]{9}. Indeed the proof of \eqref{3.10} given below follow essentially that given in \cite[pp.~40-41]{6}. We include it here to reveal that it is naturally suggested by the work here.\end{remark}

\begin{proof} Let us first treat \eqref{3.9}. By Corollary~\ref{corollary2},
\begin{align}
\notag P_\mathcal{G}(q) & = \sum_{n\geq 0} \frac{\sum_{j\geq 0} b_\mathcal{G}(n,j)}{(q^2;q^2)_n}\\
\notag & = \sum_{n\geq 0} \frac{\sum_{j\geq 0} \bigl(b_\mathcal{G}(n;2j-1)+b_\mathcal{G}(n;2j)\bigr)}{(q^2;q^2)_n}\\
\label{3.11} & = 1+ \sum_{n\geq 1} \frac{\sum_{j\geq 0} (1+q)q^{n^2+j^2+2j}\genfrac[]{0pt}{0}{n-1}{j}_2}{(q^2;q^2)_n}.
\end{align}

Now sum the $j$-sum using the $q$-binomial theorem \cite[p.~36]{9}. Hence
\begin{align*}
P_\mathcal{G}(q) & = 1+\sum_{n\geq 1} \frac{q^{n^2}(1+q)(-q^3;q^2)_{n-1}}{(q^2;q^2)_n}\\
& = \sum_{n\geq 0} \frac{q^{n^2}(-q;q^2)_n}{(q^2;q^2)_n},
\end{align*}
and \eqref{3.9} is proved.

Suppose in \eqref{3.11}, we summed on $n$ rather than $j$. Thus
\begin{align*}
    P_\mathcal{G}(q) & = 1+\sum_{n\geq 1} \frac{q^{n^2}\left(\sum_{j\ge0}q^{j^2+2j}\genfrac[]{0pt}{0}{n-1}{j}_2+\sum_{j\ge1}q^{j^2-1}\genfrac[]{0pt}{0}{n-1}{j-1}_2\right)}{(q^2;q^2)_n}\\
    &=\sum_{n\ge0}\frac{q^{n^2}\sum_{j\ge0}^nq^{j^2}\genfrac[]{0pt}{0}{n}{j}_2}{(q^2;q^2)_n}~\text{by \cite[p. 36]{9}}\\
    &=\sum_{j\ge0}\frac{q^{j^2}}{(q^2;q^2)_j}\sum_{n\ge j}\frac{q^{n^2}}{(q^2;q^2)_{n-j}}\\
    &=\sum_{j\ge0}\frac{q^{j^2}}{(q^2;q^2)_j}\sum_{n\ge 0}\frac{q^{(n+j)^2}}{(q^2;q^2)_n}\\
    &=\sum_{j\ge0}\frac{q^{2j^2}}{(q^2;q^2)_j}(-q^{2j+1};q^2)_\infty~\text{by \cite[p. 36]{9}}\\
    &=(-q;q^2)_\infty\sum_{j\ge0}\frac{q^{2j^2}}{(-q;-q)_{2j}}.
\end{align*}
Therefore
\begin{align*}
    P_\mathcal{G}(-q^2)&=(q^2;q^4)_\infty\cdot\frac{1}{2}\sum_{j\ge0}\frac{q^{j^2}}{(q^2;q^2)_j}(1+(-1)^j)\\
    &=\frac{1}{2}(q^2;q^4)_\infty((-q;q^2)_\infty+(q;q^2)_\infty)\\
    &=\frac{1}{2}\frac{(q^2;q^4)_\infty}{(q^4;q^4)_\infty}\left((q^4;q^4)_\infty(-q;q^4)_\infty(-q^3;q^4)_\infty+(q^{-1};q^4)_\infty(q;q^4)_\infty(q^3;q^4)_\infty\right)\\
    &=\frac{1}{2}\frac{(q^2;q^4)_\infty}{(q^4;q^4)_\infty}\left(\sum_{n=-\infty}^\infty q^{2n^2-n}+sum_{n=-\infty}^\infty(-1)^n q^{2n^2-n}\right)~\text{by \cite[p. 22]{9}}\\
    &=\frac{(q^2;q^4)_\infty}{(q^4;q^4)_\infty}\sum_{n=-\infty}^\infty q^{8n^2-2n}\\
    &=\frac{(q^2;q^4)_\infty}{(q^4;q^4)_\infty}(q^{16};q^{16})_\infty(-q^6;q^{16})_\infty(-q^{10};q^{16})_\infty~\text{by~\cite[p. 22]{9}}
\end{align*}
Hence
\begin{align*}
    P_\mathcal{G}(-q^2)&=\frac{(-q;q^2)_\infty}{(q^2;q^2)_\infty}(q^8;q^8)_\infty(-q^3;q^8)_\infty(-q^5;q^8)_\infty\\
    &=\frac{(q^2;q^4)_\infty(q^8;q^8)_\infty(q^3;q^8)_\infty(q^5;q^8)_\infty}{(q;q)_\infty}\\
    &=\frac{1}{(q;q^8)_\infty(q^4;q^8)_\infty(q^7;q^8)_\infty}
\end{align*}
as desired.
\end{proof}
\section{Schur's 1926 Theorem}\label{section4}
Here is the theorem in question \cite{23}.
\begin{Stheorem}
The number of partitions of $n$ in which the parts are $\equiv\pm1\pmod6$ equals the number of partitions of $n$ in which the parts differ by at least 3 and at least 4 if one of the parts in question is divisible by 3.
\end{Stheorem}
In light of the fact that
\begin{align}\label{4.1}
   \prod_{n=1}^\infty(1+q^{3n-1})(1+q^{3n-2})=\prod_{n=1}^\infty\frac{(1-q^{6n-2})(1-q^{6n-4})}{(1-q^{3n-1})(1-q^{3n-2})}=\prod_{n=1}^\infty\frac{1}{(1-q^{6n-1})(1-q^{6n-5})},
\end{align}
we see that the first class of partitions in Schur's theorem may be replaced by partitions into distinct non-multiples of 3.

Indeed, this revision of Schur may be refined as follows (an idea first effectively considered in \cite{2}):

\begin{RStheorem}
The generating function for partitions in which there are $m$ parts $\equiv0,1\pmod3$ and $n$ parts $\equiv0,2\pmod3$ and the difference conditions in Schur's original theorem hold is the coefficient of $u^mv^n$ in
\[\prod_{j=1}^\infty(1+uq^{3n-2})(1+vq^{3n-1}).\]
\end{RStheorem}
Let $\PS$
~denote the class of all partitions satisfying the difference conditions in Schur's theorem.
\begin{theorem}\label{theorem6}
$\PS$ is an SIP class.
\end{theorem}
\begin{proof}
$\PS$ is the instance of Theorem \ref{theorem1} with $k=3,\{c_1,c_2,c_3\}=\{1,2,3\}$, and $\{d_1,d_2,d_3\}=\{3,3,4\}$.
\end{proof}
In the remainder of this section, we shall first determine an explicit formula for the generating functions associated with $\BS$. Then we will apply Corollary \ref{corollary2} to prove the Refinement of Schur's Theorem.
\begin{theorem}\label{theorem7}
Let $b_\mathcal{S}(u,v,n,h)~(=b_\mathcal{S}(n,h))$ be the generating function for the partitions in $\BS$ (with $u$ marking parts $\equiv0,1\pmod3$ and $v$ marking parts $\equiv0,2\pmod3$). Then for $n>1$ and $h\ge0$,
\begin{align}
    \label{4.2}b_\mathcal{S}(n,3n+3h-1)&=\sum_{j\ge0}^{n-h}\sum_{i=0}^h v^{n-j}u^{j+h-i}q^{n(3n+1)/2+h(3h+5)/2+i(3i+1)/2}\\
    \notag&\times q^{-j}\genfrac[]{0pt}{0}{n-j-1}{h}_3\genfrac[]{0pt}{0}{j+h-i}{h}_3\genfrac[]{0pt}{0}{h}{i}_3,\\
    \label{4.3}b_\mathcal{S}(n,3n+3h-2)&=(1+uq)\sum_{j\ge0}^{n-h}\sum_{i=0}^h v^{n-j}u^{j+h-i}q^{n(3n+1)/2+h(3h+5)/2+i(3i-5)/2}\\
    \notag&\times q^{-j}\genfrac[]{0pt}{0}{n-j-1}{h}_3\genfrac[]{0pt}{0}{j+h-i}{h}_3\genfrac[]{0pt}{0}{h-1}{i-1}_3,\\
    \label{4.4}b_\mathcal{S}(n,3n+3h)&=uqb_\mathcal{S}(n,3n-3h-1),
\end{align}
and for $n=1$,
\begin{align}
    \label{4.5}b_\mathcal{S}(1,h)&=\left\{\begin{array}{cc}
        uq&\text{if}~h=1\\
        vq^2&\text{if}~h=2\\
        uvq^3&\text{if}~h=3\\
        0&\text{otherwise.}
    \end{array}\right.
\end{align}
where
\[\genfrac[]{0pt}{0}{A}{B}_j=\left\{\begin{array}{cc}
        0&\text{if}~B>A~\text{or}~B<0\\
        (q^j;q^j)_A//
        (q^j;q^j)_B(q^j;q^j)_{A-B}&\text{otherwise.}
    \end{array}\right.\]
\end{theorem}
\begin{proof}
We let
\begin{align}
    \label{4.6}\chi_3(n)=\left\{\begin{array}{cc}
        1&\text{if}~3\mid n\\
        0&\text{otherwise,}
    \end{array}\right.
\end{align}
and
\begin{align}
    \label{4.7}s(h)=\left\{\begin{array}{cc}
        uv&\text{if}~3\mid h\\
        u&\text{if}~h\equiv1\pmod3\\
        v&\text{if}~h\equiv2\pmod3.
    \end{array}\right.
\end{align}
By the definition of $\BS$, we see that
\begin{align}
    \label{4.8}b_\mathcal{S}(1,h)&=\left\{\begin{array}{cc}
        uq&\text{if}~h=1\\
        vq^2&\text{if}~h=2\\
        uvq^3&\text{if}~h=3\\
        0&\text{otherwise.}
    \end{array}\right.
\end{align}
and for $n=1$,
\begin{align}
    \label{4.9}b_\mathcal{S}(n,h)&=s(h)q^h\sum_{j=3}^5b_{\mathcal{S}}(n-1,h-j-\chi_3(h)).
\end{align}
Now \eqref{4.4} follows directly by comparing \eqref{4.9} with $h\to3n+3h-1$, and \eqref{4.5} is a restatement of \eqref{4.8}.

To prove \eqref{4.2} and \eqref{4.3}, we need only show that the coefficient of $v^{n-j}u^{j+h-i}$ on both sides of \eqref{4.9} is identical, when the $b_\mathcal{S}(n,h)$ are replaced by the corresponding right hand sides of \eqref{4.2} and \eqref{4.3}.

We begin with \eqref{4.9} when $n\to 3n+3h-1$. To make clear what we are doing, we write the right hand side of \eqref{4.2} as
\begin{align}
    \label{4.10}\sum_{j=0}^{n-h}\sum_{i=0}^h b_1(n,h,i,j)v^{n-j}u^{j+h-i}
\end{align}
and the right hand side of \eqref{4.3} as
\begin{align}
    \label{4.11}(1+uq)\sum_{j=0}^{n-h}\sum_{i=0}^h b_2(n,h,i,j)v^{n-j}u^{j+h-i}.
\end{align}
Subtracting these expressions into the right hand side of \eqref{4.9} with $h$ replaced by $3n+3h-1$, we have
\begin{align}
    \label{4.12}vq^{3n+3h-1}&\left(\sum_{j=0}^{n-1-h}\sum_{i=0}^h b_1(n-1,h,i,j)v^{n-1-j}u^{j+h-i}\right.\\
    \notag&+(1+uq)\sum_{j=0}^{n-1-h}\sum_{i=0}^h b_2(n-1,h,i,j)v^{n-1-j}u^{j+h-i}\\
    \notag&\left.+uq\sum_{j=0}^{n-h}\sum_{i=0}^{h-1}
    \notag b_1(n-1,h,i,j)v^{n-1-j}u^{j+h-i}\right),
\end{align}
the coefficient of $v^{n-j}u^{j+h-i}$ in \eqref{4.12} is
\begin{align}\label{4.13}
q^{3n+3h+1}&\left(b_1(n-1,h,i,j)+b_2(n-1,h,i,j)\right.\\
\notag&\left.+qb_2(n-1,h,i+1,j)+qb_1(n-1,h-1,i,j)\right)
\end{align}
and this last expression simplifies to $b_1(n-1,h,i,j)$ through three applications of one or the other of the standard $q$-binomial recurrences \cite[p. 35]{9} reiterated here:
\begin{align}
\label{4.14}\genfrac[]{0pt}{0}{A}{B}_j&=\genfrac[]{0pt}{0}{A-1}{B-1}_j+q^{jB}\genfrac[]{0pt}{0}{A-1}{B}_j,
\end{align}
and
\begin{align}
\label{4.15}\genfrac[]{0pt}{0}{A}{B}_j&=\genfrac[]{0pt}{0}{A-1}{B}_j+q^{j(A-B)}\genfrac[]{0pt}{0}{A-1}{B-1}_j.
\end{align}
Thus \eqref{4.9} is established for $h\to3n+3n-1$.

Finally we consider \eqref{4.9} with $h\to3n+3h+1$ in which case the assertion becomes
\begin{align}
    \label{4.16}(1+uq)&\sum_{i=0}^{n-h-1}\sum_{j=0}^{h+1} b_2(n,h+1,i,j)v^{n-j}u^{j+h-i+1}.\\
	\notag&=uq^{3n+3h+1}\left((1+uq)\sum_{i=0}^{n-h-1}\sum_{j=0}^{h} b_2(n-1,h+1,i,j)v^{n-1-j}u^{j+h-i+1}\right.\\
	\notag&\left.+(1+uq)\sum_{i=0}^{n-h-1}\sum_{j=0}^{h} b_1(n-1,h,i,j)v^{n-1-j}u^{j+h-i}\right).
\end{align}
Cancelling the $u(1+uq)$ from both sides of \eqref{4.16}, we see that we need to establish that the coefficients of $v^{n-j}u^{j+h-i}$ are identical on each side.

On the right hand side the coefficient is
\begin{align*}
&q^{3n+3h+1}(b_2(n-1,h+1,i,j)+b_1(n-1,h,i-1,j-1))\\
&=q^{n(3n+1)/2+(h+1)(3h+8)/2+i(3i-5)/2}\\
&\times\left(q^{3h+3}\genfrac[]{0pt}{0}{n-1-j}{h}_3\genfrac[]{0pt}{0}{j+h-i}{h+1}_3\genfrac[]{0pt}{0}{h}{i-1}_3\right.\\
&\left.+\genfrac[]{0pt}{0}{n-1-j}{h}_3\genfrac[]{0pt}{0}{j+h-i}{h}_3\genfrac[]{0pt}{0}{h}{i-1}_3\right)\\
&=b_2(n,h+1,i,j)
\end{align*}
by \eqref{4.14}.

Thus we have fulfilled the defining recurrence and initial conditions to establish the formulas \eqref{4.2} and \eqref{4.3}. This concludes the proof of Theorem \ref{theorem7}.
\end{proof}
We shall conclude this section by proving the Refinement of Schur's Theorem combining Corollary \ref{corollary2} with Theorem \ref{theorem7}.

For the sake of brevity, we write
\begin{equation}
	\label{4.17}S_1(n,h)=b_{\mathcal{S}}(n,3n+3h-1)
\end{equation}
and
\begin{equation}
	\label{4.18}S_2(n,h)=b_{\mathcal{S}}(n,3n+3h-2).
\end{equation}
We note that by \eqref{4.4},
\begin{equation}
	\label{4.19}uqS_2(n,h)=b_{\mathcal{S}}(n,3n+3h).
\end{equation}
To make the final theorem of this section readable, we first prove three lemmas.
\begin{lemma}\label{lemma8}
\begin{align}
	\label{4.20}(1+uq)S_1(n,h)+S_2(n,h+1)&=\sum_{j=0}^{n-h}\sum_{i=-1}^h v^{n-j}u^{j+h-i}q^{n(3n+1)/2+h(3h+8)/2+i(3i+1)/2-j}\\
	\notag&\times\genfrac[]{0pt}{0}{n-1-j}{h}_3\genfrac[]{0pt}{0}{j+h-i}{j}_3\genfrac[]{0pt}{0}{j+1}{i+1}_3
\end{align}
\end{lemma}
\begin{proof}
The coefficient of $v^{n-j}u^{j+h-i}$ in
\[(1+uq)s_1(n,h)+s_2(n,h+1)\]
is
\begin{align*}
    &q^{n(3n+1)/2+h(3h+5)/2+i(3i+1)/2-j}\\
    &\times\left(\genfrac[]{0pt}{0}{n-1-j}{h}_3\genfrac[]{0pt}{0}{j+h-i}{h}_3\genfrac[]{0pt}{0}{h}{i}_3\right.\\
    &+q^{3i+3}\genfrac[]{0pt}{0}{n-1-j}{h}_3\genfrac[]{0pt}{0}{j+h-i-1}{h}_3\genfrac[]{0pt}{0}{h}{i+1}_3\\
    &\left.+q^{3h+3}\genfrac[]{0pt}{0}{n-1-j}{h}_3\genfrac[]{0pt}{0}{j+h-i}{h+1}_3\genfrac[]{0pt}{0}{h}{i}_3\right.\\
    &\left.+q^{3i+3h+6}\genfrac[]{0pt}{0}{n-1-j}{h}_3\genfrac[]{0pt}{0}{j+h-i-1}{h+1}_3\genfrac[]{0pt}{0}{h}{i+1}_3\right)\\
    &=q^{n(3n+1)/2+h(3h+5)/2+i(3i+1)/2-j}\\
    &\times\left(\genfrac[]{0pt}{0}{n-1-j}{h}_3\genfrac[]{0pt}{0}{j+h-i+1}{h+1}_3\genfrac[]{0pt}{0}{h}{i}_3\right.\\
    &\left.+q^{3i+3}\genfrac[]{0pt}{0}{n-1-j}{h}_3\genfrac[]{0pt}{0}{j+h-i}{h+1}_3\genfrac[]{0pt}{0}{h}{i+1}_3\right)~(\text{by \eqref{4.14} twice})\\
    &=q^{n(3n+1)/2+h(3h+5)/2+i(3i+1)/2-j}\genfrac[]{0pt}{0}{n-1-j}{h}_3\\
    &\times\frac{(q^3;q^3)_{j+h-i}(q^3;q^3)_{h}}{(q^3;q^3)_{h+1}(q^3;q^3)_{j-i}(q^3;q^3)_{i+1}(q^3;q^3)_{h-i}}\\
    &\times\left((1-q^{3(j+h-i+1)})(1-q^{3(i+1)})+q^{3i+3}(1-q^{3(j-i)})(1-q^{3(h-i)})\right)\\
    &=q^{n(3n+1)/2+h(3h+5)/2+i(3i+1)/2-j}\times\genfrac[]{0pt}{0}{n-1-j}{h}_3\genfrac[]{0pt}{0}{j+h-i}{j}_3\genfrac[]{0pt}{0}{j+1}{i+1}_3,
\end{align*}
which is the coefficient asserted in \eqref{4.20}.
\end{proof}
\begin{lemma}\label{lemma9}
\begin{equation}
\label{4.21}\sum_{h=0}^r\genfrac[]{0pt}{0}{s-1}{h}_3\genfrac[]{0pt}{0}{n+1}{r-h}_3q^{3h^2+3h(n+1-r)}=\genfrac[]{0pt}{0}{n+s}{r}_3
\end{equation}
\end{lemma}
\begin{proof}
This is an instance of the q-Chu-Vandermonde summation \cite[p. 37]{3}, eq. (3,3.10), $q\to q^3,n\to n+1,m\to s-1,h\to r$, and $k\to r-k$.
\end{proof}
\begin{lemma}\label{lemma10}
\begin{equation}
\label{4.22}\sum_{n\ge0}\genfrac[]{0pt}{0}{r}{n}_3\genfrac[]{0pt}{0}{n+s}{r}_3\frac{q^{3n^2+3n(s-r)}}{(q^3;q^3)_{n+s}}=\frac{1}{(q^3;q^3)_r(q^3;q^3)_s}
\end{equation}
\end{lemma}
\begin{proof}
\begin{align*}
	\sum_{n\ge0}\genfrac[]{0pt}{0}{r}{n}_3\genfrac[]{0pt}{0}{n+s}{r}_3 \frac{q^{3n^2+3n(s-r)}}{(q^3;q^3)_{n+s}}&=\frac{1}{(q^3;q^3)_r(q^3;q^3)_{s-r}}\lim_{\tau\to0}\sum_{n=0}^r \frac{(q^{-3r};q^3)_n(q^{3s+3}/\tau;q^3)_n \tau^n} {(q^3;q^3)_n(q^{3(s-r+1)};q^3)_n}\\
	&=\frac{1}{(q^3;q^3)_r(q^3;q^3)_{s-r}}\cdot\frac{1}{(q^{3(s-r+1)};q^3)_r}\\
    &\text{(by \cite[p. 38 eq. (3.3,12)]{9})},q\to q^3,N=r,a\to\infty,b=q^{3s+3}/\tau,\\c=q^{3(s-r+1)}\\
    &=\frac{1}{(q^3;q^3)_r(q^3;q^3)_s}
\end{align*}
\end{proof}
\begin{theorem}\label{theorem11}
\begin{equation}\label{4.23}
	\sum_{n\ge0}\frac{b_{\mathcal{S}}(n;q)}{(q^3;q^3)_n}=(-uq;q^3)_\infty(-vq;q^3)_\infty.
\end{equation}
\end{theorem}
\begin{remark}
This is the restatement of the refinement of Schur’s theorem. It will appear in the following proof that the $b_\mathcal{S}(n;q)$ are infinite sums, but this is only a convenience of notation. Thus
\begin{align*}
	\sum_{n\ge0}\frac{B_{\mathcal{S}}(n;q)}{(q^3;q^3)_n}&=1+\frac{uq+vq^2+uvq^3}{1-q^3}\\
    &+\frac{u^2q^5+uvq^6+(u^2v+v^2)q^7+uv^2q^8+uvq^9+u^2vq^{10}+v^2uq^{11}+u^2v^2q^{12}}{(1-q^3)(1-q^6)}\\
    &+\cdots
\end{align*}
\end{remark}
\begin{proof}
First we note that
\[(-uq;q^3)_\infty(-vq^2;q^3)_\infty=\sum_{r=0}^\infty\frac{u^rq^{\frac{1}{2}r(3r-1)}}{(q^3;q^3)_r}\sum_{s=0}^\infty\frac{v^sq^{\frac{1}{2}s(3s+1)}}{(q^3;q^3)_s}.\]
Hence the coefficient of $u^rv^s$ on the right hand of \eqref{4.23} is
\begin{equation}\label{4.24}
\frac{q^{r(3r-1)/2+s(3s+1)/2}}{(q^3;q^3)_r(q^3;q^3)_s}.
\end{equation}
To complete the proof we must evaluate the coefficient of $u^rv^s$ on the left side of \eqref{4.23}.

Now
\begin{align*}
	\sum_{n\ge0}\frac{b_\mathcal{S}(n,q)}{(q^3;q^3)_n} &=\sum_{\substack{n\ge0\\h\ge-1}}\frac{(1+uq)s_1(n,h)+s_2(n,h+1)}{(q^3;q^3)_n}\\
	&=1+\sum_{\substack{n\ge0\\h\ge-1}}\frac{1}{(q^3;q^3)_n}\sum_{j=0}^{n-h}\left(\sum_{i=-1}^h v^{n-j}u^{j+h-i}q^{f(n,h,i,j)}\right.\\
	&\left.\times\genfrac[]{0pt}{0}{n-j-1}{h}_3\genfrac[]{0pt}{0}{j+h-i}{j}_3\genfrac[]{0pt}{0}{j+1}{i+1}_3\right),
\end{align*}
by Lemma \ref{lemma8}, with
\[f(n,h,i,j)=n(3n+1)/2+h(3h+5)/2+i(3i+1)/2-j.\]
So to get the coefficient of $u^rv^s$, we need $j=n-s,i=n-s-r+h$. Thus the coefficient of $u^rv^s$ on the left side of \eqref{4.23} is
\[\sum_{\substack{n\ge0\\h\ge-1}}\genfrac[]{0pt}{0}{s-1}{h}_3\genfrac[]{0pt}{0}{r}{n-s}_3\genfrac[]{0pt}{0}{n-s+1}{r-h}_3 \times\frac{q^{f(n,h,n-s-r+h,n-s)}}{(q^3;q^3)_n}.\]
Now the sum on $h$ turns out to be the sum on the left side of \eqref{4.21}. Hence by Lemma \ref{lemma9}, the above sum reduces to
\[\sum_{n\ge0}\genfrac[]{0pt}{0}{r}{n}_3\genfrac[]{0pt}{0}{n+s}{r}_3\frac{q^{s(3s+1)/2+r(3r-1)/2+3n^2+3n(s-r)}}{(q^3;q^3)_{n+s}}=\frac{q^{s(3s+1)/2+r(3r-1)/2}}{(q^3;q^3)_t(q^3;q^3)_s}\]
by Lemma \ref{lemma10} which is exactly the expression in \eqref{4.24}. Thus Theorem \ref{theorem11} is proved.
\end{proof}
\section{Glasgow Mod 8}\label{section5}

H. G\"ollnitz \cite{16} \cite{17} provided four partition identities related to partitions whose parts are restricted to certain residue classes modulo 8. Two of these theorems were independently discovered by B Gordon \cite{19}\cite{20} and have been given the name G\"ollnitz-Gordon, as mentioned previously.

Lesser known is the following theorem which first appeared in the Glasgow Mathematics Journal in 1967 \cite[p. 127]{3}:
\begin{Glasgow}
Let $A(n)$ denote the number of partitions of $n$ into parts congruent to 0, 2, 3, 4, or $7\pmod 8$. Let $B(n)$ denote the number of partitions of $n$ in which all parts are $\ge2$ and each odd part is at least 3 larger than any part not exceeding it. Then for $n\ge0$,
\[A(n)=B(n).\]
For example, $A(10)=8$ enumerating
\[10, 8+2, 7+3, 4+3+3, 4+4+2,4+2+2+2, 3+3+2+2, 2+2+2+2+2,\]
and $B(10)=8$ enumerating 
\[10, 8+2, 7+3, 6+4, 6+2+2, 4+4+2, 4+2+2+2, 2+2+2+2+2.\]
\end{Glasgow}

A natural bijective proof appears in \cite{4}.

We have chosen to consider this theorem owing to the fact that it has never appeared as a direct consequence of a series-product identity. Indeed, the relevant identity turns out to be
\begin{equation}\label{5.1}
	1+\frac{q^2+q^3}{1-q^2}+\sum_{n=2}^\infty\frac{(-q^3;q^4)_{n-1}q^{2n}(1+q^{2n-1})}{(q^2;q^2)_n}=\prod_{\substack{n=1\\n\not\equiv1,5,6\pmod 8}}\frac{1}{1-q^n}.
\end{equation}

We shall first prove that \eqref{5.1} is valid. We shall then prove that the left side of \eqref{5.1} is an instance of Corollary \ref{corollary2}.
\begin{theorem}\label{theorem12}
Equation \eqref{5.1} is valid.
\end{theorem}
\begin{proof}
For $N\ge1$,
\begin{equation}\label{5.2}
1+\frac{q^2+q^3}{1-q^2}+\sum_{n=2}^\infty\frac{(-q^3;q^4)_{n-1}q^{2n(1+q^{2n-1})}}{(q^2;q^2)_n}=\frac{(-q^3;q^4)_N}{(q^2;q^2)_N}.
\end{equation}
This follows by mathematical induction. For $N=1$,
\[1+\frac{q^2+q^3}{1-q^2}=\frac{1+q^3}{1-q^2}.\]
Generally,
\begin{align*}
	\frac{(-q^3;q^4)_N}{(q^2;q^2)_N}-\frac{(-q^3;q^4)_{N-1}}{(q^2;q^2)_{N-1}}&= \frac{(-q^3;q^4)_{N-1}}{(q^2;q^2)_N}\left((1+q^{4N-1})-(1-q^{2N})\right)\\
&=\frac{(-q^3;q^4)_{N-1}q^{2N}(1+q^{2N-1})}{(q^2;q^2)_N},
\end{align*}
which is the $N^{th}$ term of the left-hand side. The result then follows by induction.

Now let $N\to\infty$ in \eqref{5.2}. The left side converges to the left side of \eqref{5.1}, and
\begin{align*}
	\frac{(-q^3;q^4)_\infty}{(q^2;q^2)_\infty}&=\frac{(q^6;q^8)_\infty}{(q^2;q^2)_\infty(q^3;q^4)_\infty}\\
&=\frac{1}{(q^2,q^3,q^4,q^7,q^8;q^8)_\infty}~\left(=\sum_{n\ge0}A(n)q^n\right).
\end{align*}
\end{proof}
\begin{lemma}\label{lemma13}
Let $\E$ denote the class of partitions related to $B(n)$. Then $\E$ is an SIP class of modulus 2.
\end{lemma}
\begin{proof}
This follows immediately from Theorem \ref{theorem1} with $k=2,d_1=3,d_2=0,c_1=3,c_2=2$.
\end{proof}
\begin{lemma}\label{lemma14}
Let $b_\E(n,h)$ be the generating function for the partitions in $B_\E$ with $n$ parts and largest part equal to $h$. Then
\begin{equation}\label{5.3}
b_\E(1,h)=\left\{\begin{array}{cc}
q^2&\text{if}~h=2\\
q^3&\text{if}~h=3\\
0&\text{otherwise}
\end{array}\right.,
\end{equation}
and for $n>1,h>0$,
\begin{equation}\label{5.4}
b_\E(n,4h+1)=q^{2n+2h^2+h}\genfrac[]{0pt}{0}{n-2}{h-1}_4,
\end{equation}
\begin{equation}\label{5.5}
b_\E(n,4h)=q^{2n+2h^2+h}\genfrac[]{0pt}{0}{n-2}{h-1}_4,
\end{equation}
\begin{equation}\label{5.6}
b_\E(n,4h-1)=q^{4n+2h^2-3h}\genfrac[]{0pt}{0}{n-2}{h-2}_4,
\end{equation}
\begin{equation}\label{5.7}
b_\E(n,4h-2)=q^{2n-3+2h^2+h}\genfrac[]{0pt}{0}{n-2}{h-1}_4.
\end{equation}
\end{lemma}
\begin{proof}
First we see that \eqref{5.3} is immediate by inspection. Next we note that the two part partitions in $B_\E$ are $2+2,2+5,3+4,3+7$. Thus
\begin{equation}\label{5.8}
b_\E(2,h)=\left\{\begin{array}{cc}
q^4&\text{if}~h=2\\
q^7&\text{if}~h=4\\
q^7&\text{if}~h=5\\
q^{10}&\text{if}~h=7\\
0&\text{otherwise}
\end{array}\right.,
\end{equation}
and inspection reveals that \eqref{5.4}-\eqref{5.7} are valid for $n=2$.

Now as in the previous sections, we see that
\begin{equation}\label{5.9}
b_\E(n,h)=\left\{\begin{array}{cc}
q^h(b_\E(n-1,h)+b_\E(n-1,h-1))&\text{if $h$ is even}\\
q^h(b_\E(n-1,h-3)+b_\E(n-1,h-4))&\text{if $h$ is odd}
\end{array}\right..
\end{equation}
All that remains is to show that the right hand sides of \eqref{5.4}-\eqref{5.7} satisfy the defining recurrence \eqref{5.9}. Each is proved using instances of \eqref{4.14} or \eqref{4.15}. We shall do one case which is typical. When $h\equiv1\pmod4$, equation \eqref{5.9} asserts
\begin{equation}\label{5.10}
b_\E(n,4h+1)=q^{4h+1}\left(b_\E(n-1,4h-2)+b_\E(n-1,4h-3)\right).
\end{equation}
If we replace the $b_\E(n$, \textunderscore$)$ by the relevant right side of \eqref{5.4}-\eqref{5.7}, the assertion is:
\begin{align*}
	q^{4h+1}&\left(q^{2n-5+2h^2+h}\genfrac[]{0pt}{0}{n-2}{h-1}_4+q^{2n-2+2(h-1)^2+h-1}\genfrac[]{0pt}{0}{n-2}{h-2}_4\right)\\
&=q^{2n+2h^2+h}\left(q^{4h-4}\genfrac[]{0pt}{0}{n-2}{h-1}_4+{\genfrac[]{0pt}{0}n-2}{h-2}_4\right)\\
&=q^{2n+2h^2+h}\genfrac[]{0pt}{0}{n-2}{h-1}_4~\text{(by \eqref{4.14}),}
\end{align*}
and this is exactly the recurrence \eqref{5.10}.
\end{proof}
\begin{lemma}\label{lemma15}
\begin{equation}\label{5.11}
\sum_{h\ge0}b_\E(n,4h+1)=q^{2n+3}(-q^7;q^4)_{n-2},
\end{equation}
\begin{equation}\label{5.12}
\sum_{h\ge0}b_\E(n,4h)=q^{4n-1}(-q^7;q^4)_{n-2},
\end{equation}
\begin{equation}\label{5.13}
\sum_{h\ge0}b_\E(n,4h-1)=q^{4n+2}(-q^7;q^4)_{n-2},
\end{equation}
\begin{equation}\label{5.14}
\sum_{h\ge0}b_\E(n,4h-2)=q^{2n}(-q^7;q^4)_{n-2}.
\end{equation}
\end{lemma}
\begin{proof}
Each of these four assertions is an instance of the $q$-binomial theorem \cite[p. 36]{9} applied to the corresponding equation in Lemma \ref{lemma14}. We prove \eqref{5.10} as typical.

By \eqref{5.4} and the $q$-binomial theorem,
\begin{align*}
	\sum_{h\ge0}b_\E(n,4h+1)&=\sum_{h\ge0}q^{2n+2h^2+h}\genfrac[]{0pt}{0}{n-2}{h-1}_4\\
&=\sum_{h\ge0}q^{2n+2(h+1)^2+h+1}\genfrac[]{0pt}{0}{n-2}{h}_4\\
&=q^{2n+3}(-q^7;q^4)_{n-2}.
\end{align*}
\end{proof}
\begin{theorem}\label{theorem16}
\begin{equation}\label{5.15}
\sum_{n\ge0}\frac{b_\E(n)}{(q^2;q^2)_n}=1+\frac{q^2+q^3}{1-q^2}+\sum_{n\ge2}\frac{(-q^3;q^4)_{n-1}q^{2n}(1+q^{2n-1})}{(q^2;q^2)_n}.
\end{equation}
\end{theorem}
\begin{proof}
\begin{align*}
	\sum_{n\ge0}\frac{b_\E(n)}{(q^2;q^2)_n}&=1+\frac{q^2+q^3}{1-q^2}+\sum_{n\ge2}\sum_{h\ge1}\frac{b_\E(n,h)}{(q^2;q^2)_n}\\
	&=1+\frac{q^2+q^3}{1-q^2}+\sum_{n\ge2}\sum_{h\ge1}\frac{b_\E(n,4h+1)+b_\E(n,4h)+b_\E(n,4h-1)+b_\E(n,4h-2)}{(q^2;q^2)_n}\\
	&=+\sum_{n\ge2}\frac{(-q^7;q^4)_{n-2}(q^{2n+3}+q^{4n-1}+q^{4n+2}+q^{2n})}{(q^2;q^2)_n}~\text{(by Lemma \ref{lemma14})}\\
	&=1+\frac{q^2+q^3}{1-q^2}+\sum_{n\ge2}\frac{(-q^7;q^4)_{n-2}q^{2n}(1+q^3)(1+q^{2n-1})}{(q^2;q^2)_n}\\
    &=1+\frac{q^2+q^3}{1-q^2}+\sum_{n\ge2}\frac{(-q^3;q^4)_{n-1}q^{2n}(1+q^{2n-1})}{(q^2;q^2)_n}.
\end{align*}
\end{proof}
\begin{corollary}\label{corollary17}
The Glasgow Mod 8 Theorem is true.
\end{corollary}
\begin{proof}
This follows by comparing \eqref{5.1} with \eqref{5.14} and invoking Corollary \ref{corollary2} with $\mathcal{P}=\E$
\end{proof}
\section{Partitions with $n$ copies of $n$}\label{section6}
The basic idea epitomized by Theorem \ref{theorem1} is actually applicable in a broader context. In this section we shall describe its application to partitions with ``$n$ copies of $n$’’ \cite{1}.

This subject considers partitions taken from the set $M$ of ordered pairs of positive integers with the second entry not exceeding the first entry. A partition with $n$ copies of $n$ of the positive integer $v$ is a finite collection of elements of $M$ wherein the first element of the ordered pairs sum to $v$. For example, there are six partitions of $3$ with $n$ copies of $n$:
\[3_1,3_2,3_3,2_2+1_1, 2_1+1_1, 1_1+1_1+1_1.\]

As was noted in \cite{1}, there is a bijection between partitions with $n$ copies of $n$ and plane partitions.

Most important for our current considerations is the weighted difference between two elements of $M$. Namely, we define $((m_i-n_j))$, the weighted difference of $m_i$ and $n_j$, as follows:
\[((m_i-n_j))=m-n-i-j.\]

The main point of \cite{1} was to prove the following two results.
\begin{theorem}\label{theorem18}
\cite[p. 41]{1} The partitions of $v$ with $n$ copies of $n$ wherein each pair of parts has positive weighted difference are equinumerous with the ordinary partitions of $v$ into parts $\equiv0,\pm4\pmod{10}$.
\end{theorem}
\begin{theorem}\label{theorem19}
\cite[p. 41]{1} The partitions of $v$ with $n$ copies of $n$ wherein each pair of parts has nonnegative weighted difference are equinumerous with the ordinary partitions of $v$ into parts $\equiv0,\pm6\pmod{14}$.
\end{theorem}
These two theorems are special cases of a general theorem proved in \cite[Th. 3, p. 42]{1}.
The proofs relied on bijection between the partitions in question and the results which provide several Rogers-Ramanujan type theorems concerning partitions with specified hook differences.

Our object here is to reveal a completely different path to proof by using an adaptation of theorem \ref{theorem1}.

Let $\beta_r(m;q)$ denote the generating function for partitions with $n$ copies of $n$ where the weighted difference between successive parts (written in lexicographic ascending order) is exactly $r$, there are exactly $m$ parts, and the smallest part is of the form $i_i$.
\begin{theorem}\label{theorem20}
\begin{equation}\label{6.1}
	\beta_r(m,q)=\frac{q^{m^2+r\binom{m}{2}}}{(q;q^2)_m}.
\end{equation}
\end{theorem}
As we will see, Theorems \ref{theorem18} and \ref{theorem19} follow from Theorem \ref{theorem20} plus Lemma \ref{lemma22} via two identities given in L. J. Slater’s compendium \cite[eqs (46) and (61)]{24}
\begin{equation}\label{6.2}
\sum_{n\ge0}\frac{q^{n(3n-1)/2}}{(q;q)_n(q;q^2)_n}=\prod_{\substack{n=1\\n\not\equiv0,\pm4\pmod{10}}}\frac{1}{1-q^n},
\end{equation}
and
\begin{equation}\label{6.3}
\sum_{n\ge0}\frac{q^{n^2}}{(q;q)_n(q;q^2)_n}=\prod_{\substack{n=1\\n\not\equiv0,\pm6\pmod{14}}}\frac{1}{1-q^n}.
\end{equation}
In addition, the case $r=-1$ is related to the Slater identity \cite[p. 160, eq. (81)]{24}
\begin{equation}\label{6.4}
\sum_{n\ge0}
\frac{q^{\binom{n+1}{2}}}{(q;q)_n(q;q^2)_n}=\prod_{n=1}^\infty(1+q^{7n})\prod_{m=1}^\infty\frac{1}{(1-q^{14m-3})(1-q^{14m-11})}\times\prod_{\substack{n=1\\n\equiv\pm2,3,4\pmod{14}}}^\infty\frac{1}{1-q^n}.
\end{equation}
The right hand side of \eqref{6.4} is easily seen to be the generating function for $C(n)$, the number of partitions in which multiples of 7 are not repeated, all other parts are $\equiv\pm2,\pm3,\pm4\pmod{14}$ and parts $\equiv\pm3\pmod{14}$ appear in two colors. This observation together with \eqref{6.4} establishes the following result.
\begin{theorem}\label{theorem21}
The number of partitions of $v$ with $n$ copies of $n$ wherein successive parts have weighted difference $\ge-1$ equals $C(n)$.
\end{theorem}
We shall not require the full generality of Theorem \ref{theorem1} for our application of the SIP idea to partitions with $n$ copies of $n$. Indeed we only need something analogous to the three classical examples provided initially in Section \ref{section2}.
\begin{lemma}\label{lemma22}
Let $r\ge-1$. Suppose $\pi$ is a partition with $n$ copies of $n$ with the parts written in ascending lexicographic order (i.e. $m_i>n_j$ if $m>n$ or $m=n$ and $i>j$).
Assume that the weighted difference between successive parts is $\ge r$. Then if $\pi$ has $h$ parts
\[m_i+n_j+o_h+\cdots+t_k,\]
there is a unique ordinary partition with $h$ nonnegative parts in non-decreasing order
\[\psi_1+\psi_2+\cdots+\psi_h\]
and a unique partition $\bar\pi$ with $n$ copies of $n$
\[\bar m_i+\bar n_j+\bar o_h+\cdots+\bar t_k,\]
where $\bar m_i=i_i$ and the successive weighted differences are all equal to $r$, and
\begin{align*}
	m_i&=\bar m_i+\psi_1=(\bar m+\psi_1)_i\\
	n_j&=\bar n_j+\psi_2=(\bar n+\psi_2)_j\\
	&\vdots
	t_k&=\bar t_k+\psi_h=(\bar t+\psi_h)_k.
\end{align*}
\end{lemma}
\begin{remark}
Note that the subscripts for the original $\pi$ are identical with the subscript set for $\bar\pi$.
\end{remark}
\begin{proof}
We begin by noting that the subscript tuple $(i,j,h,\dots,k)$ uniquely defines the $\bar m_i,\bar n_j,\dots$ as follows:
\begin{align*}
	\bar m_i&=i_i\\
	\bar n_j&=(j+2i+r)_j\\
	\bar o_h&=(h+2i+2j+2r)_j
	&\vdots
	\bar t_k&=(t+\cdots+2h+2i+2j+(p-1)r)_k,
\end{align*}
where $p$ is the number of parts of the partition. Note that the weighted difference between successive terms in this sequence is always $r$.

Now we uniquely construct the $\psi$ as follows. We begin with $\psi_1$:
\[\psi_1=m-i,\]
and since the first subscript is $i$, we know that $m$ must be $\ge i$, so $\psi_1\ge0$. Next we define
\[\psi_2=n-j-2i-r.\]
Clearly $\psi_2$ is unique. Is $\psi_2\ge\psi_1$? Yes, because
\[\psi_2-\psi_1=(n-j-2i-r)-(m-i)=n-m-i-j-r=((n_i-m_j))-r\ge r-r\ge0.\]
Next we define
\[\psi_3=o-h-2i-2j-2r\]
and again
\[\psi_3-\psi_2=(o-h-2i-2j-2r)-(n-j-2i-r)=o-h-n-j-r=((o_h-n_j))-r\ge r-r\ge0.\]
This continues for all the parts of $\bar\pi$, and this concludes the proof of the lemma.
\end{proof}
Thus we have established the analogous paradigm for $n$-copies of $n$ partitions that we considered for ordinary partitions.

The next step is to consider the generating function for the partitions
\[\bar m_i+\bar n_j+\cdots+\bar t_k,\]
where $\bar m_i=i_i$ and all weighted differences between successive parts in ascending order equal $r$. Call this generating function $g_r(n,m,j)$ for such partitions where the number of parts is $n$ and the largest part is $m_j$.
\begin{lemma}\label{lemma23}
\begin{equation}\label{6.5}
	g_r(n,m,j)=\left\{\begin{array}{cc}
		0&\text{if}~m<1\\
		q^m&\text{if}~n=1~\text{and}~m=j\\
		0&\text{if}~n=1~\text{and}~m\neq j\\
		q^m\sum_{i=1}^m g_r(n-1,m-j-i-r,i)&\text{otherwise.}\\
\end{array}\right.
\end{equation}
\end{lemma}
\begin{proof}
The first three lines of \eqref{6.5} are immediate because the smallest part must be of the form $m_m$.

For the last line, we see that the part $m_j$ must have directly below it a part that produces a weighted difference of $r$. Thus is the subscript is $i$, the part must be
\[(m-j-i-r)_i\]
Because
\[m-j-(m-j-i-r)-i=r.\]
Thus summing over $i$ we obtain the fourth line of \eqref{6.5}.
\end{proof}
\begin{lemma}\label{lemma24}
For $r\ge0$,
\begin{align}\label{6.6}
	g_{2r-1}(2n,2m,2j-1)&=q^2g_{2r-1}(2n,2m-1,2j)\\
	\notag&=q^{3m-j+(4r+2)n^2-(8r+2)n+3r+1}\genfrac[]{0pt}{0}{m-(2r-1)n-j+r-1}{2n-2}_2\\
	\label{6.7}g_{2r-1}(2n-1,2m,2j)&=qg_{2r-1}(2n-1,2m-1,2j-1)\\
	\notag&=q^{3m-j+(4r+2)n^2-(12r+4)n+8r+2}\genfrac[]{0pt}{0}{m-(2r-1)n-j+2r-2}{2n-3}_2\\
	\label{6.8}g_{2r}(2n,2m,2j)&=qg_{2r}(2n,2m-1,2j-1)\\
	\notag&=q^{3m-j+(4r+4)n^2-(8r+6)n+3r+2}\genfrac[]{0pt}{0}{m-2rn-j+r-1}{2n-2}_2\\
	\label{6.9}g_{2r}(2n-1,2m,2j)&=qg_{2r}(2n-1,2m-1,2j-1)\\
	\notag&=q^{3m-j+(4r+4)n^2-(12r+10)n+8r+6}\genfrac[]{0pt}{0}{m-2rn-j+2r-1}{2n-3}_2.
\end{align}
All instances of $g_r(n,m,j)$ apart from those listed in \eqref{6.5}-\eqref{6.9} are identically zero.
\end{lemma}
\begin{proof}
Let us use $\gamma_r(n,m,j)$ for the right hand sides of \eqref{6.5}-\eqref{6.9}. It is clear that the recurrence and initial conditions in Lemma \ref{lemma24} uniquely define the $g_r(n,m,j)$ polynomials.

It is easy to check directly that the two top lines of \eqref{6.5} hold for $\gamma_r(n,m,j)$. It is also a simple matter to verify that all of the instances of $\gamma_r(n,m,j)$ that should be identicaly zero are indeed that via the given recurrence.

The heart of the proof is to show that each of 8 instances for $\gamma_r(n,m,j)$ required by \eqref{6.6}-\eqref{6.9} actually fulfill the recurrence. Each one is very similar to the others, so we will do \eqref{6.6} for $\gamma_{2r-1}(2,2m,2j)$. First the case $n=1$, which asserts
\[\gamma_{2r-1}(2,2m,2j-1)=q^{3m-j-r+1}.\]
This is true because if there are just two parts where the larger lexicographically is $(2m)_{2j-1}$ and the smaller is some $M_M$ with the requirement that
\[(((2m)_{2j-1}-MM))=2r-1,\]
then $2m-(2j-1)-M-M=2r-1$. So $M=m-j-r+1$, and $2m+M=3m-j-r+1$, as required.

Next we must treat the recurrence step. We evaluate the last line in \eqref{6.5} for the $\gamma$’s,
\begin{align*}
	&q^{2m}\sum_{i=1}^{2m}g_{2r-1}(2n-1,2m-(2j-1)-i-(2r-1),i)\\
	&=q^{2m}\left(\sum_{i=1}^m g_{2r-1}(2n-1,2m-2j-2i-2r+2,2i)\right.\\
	&\left.+\sum_{i=1}^mg_{2r-1}(2n-1,(2m-2j-2i-2r+4)-1,2i-1)\right)\\
	&=q^{2m}\sum_{i=1}^m\left(q^{3(m-j-i-r+1)-i+(4r+2)n^2-(12r+4)n+8r+2}\right.\\
	&\left.\times\genfrac[]{0pt}{0}{(m-j-i-r+1)-(2r-1)n-i+2r-2}{2n-3}_2\right.\\
	&\left.+q^{3(m-j-i-r+2)-i+(4r+2)n^2-(12r+4)n+8r+1}\right.\\
	&\left.\times \genfrac[]{0pt}{0}{(m-j-i-r+2)-(2r-1)n-i+2r-2}{2n-3}_2\right)\\
	&=q^{3m-3j+5r+(4r+2)n^2-(12r+4)n+5}\sum_{i=1}^m \left(q^{2m-4i}\genfrac[]{0pt}{0}{m-j+r-1-(2r-1)n-2i}{2n-3}_2\right.\\
	&\left.+q^{2m-4i+2}\genfrac[]{0pt}{0}{m-j+r-1-(2r-1)n-(2i-1)}{2n-3}_2\right)\\
	&=q^{3m-3j+5r+(4r+2)n^2-(12r+4)n+5}\sum_{i=1}^{2m}q^{2m-2i} \genfrac[]{0pt}{0}{m-i-j+r-1-(2r-1)n}{2n-3}_2\\
	&=q^{3m-3j+5r+(4r+2)n^2-(12r+4)n+8r+5}\times q^{2(j-r+1+(2r-1)n+2n-3)}\\
    &\times\sum_{i=1}^{2m}q^{2(m-i-j+r-1-(2r-1)n-2n+3)} \genfrac[]{0pt}{0}{m-i-j+r-1-(2r-1)n}{2n-3}_2\\
	&=q^{3m-j+(4r+2)n^2-(8r+2)n+3r+1}\times \genfrac[]{0pt}{0}{m-1-j+r-(2r-1)n}{2n-2}_2\\
	&~\text{(by \cite[p. 37, eq. (3.3.9)]{9})}\\
	&=\gamma_{2r-1}(2n,2m,2j-1).
\end{align*}
The other seven recurrences and initial conditions are proved in exactly this way.
\end{proof}
We are now in a position to prove Theorem \ref{theorem20}.
\begin{proof}
By the definition of $g_r(n,m,j)$ we see that $\beta_r(n,q)=\sum_{m\ge1}\sum_{j=1}^m g_r(n,m,j)$.

There are four cases to treat: $n$ even or odd and $r$ even or odd. The cases are entirely similar, so we consider only $r$ odd and $n$ odd.
\begin{align*}
	\beta_{2r-1}(2n-1,q)&=\sum_{m\ge1}\sum_{j=1}^m g_{2r-1}(2n-1,m,j)\\
	&=\sum_{m\ge1}\sum_{j=1}^m(g_{2r-1}(2n-1,2m,2j)+g_{2r-1}(2n-1,2m-1,2j-1))\\
	&=(1+q)\sum_{m\ge1}\sum_{j=1}^mq^{3m-j+(4r+2)n^2-(12r+4)n+8r+1} \times\genfrac[]{0pt}{0}{m-(2r-1)n-j+2r-2}{2n-3}_2\\
	&=(1+q)\sum_{m\ge0}\sum_{j\ge1}q^{3(m+(2r-1)n+j-2r+2+2n-3)} \times q^{-j+(4r+2)n^2-(12r+4)n+8r+1}\genfrac[]{0pt}{0}{m+2n-3}{2n-3}_2\\
	&=(1+q)\frac{q^2}{1-q^2}\frac{1}{(q^3;q^2)_{2n-2}}\times q^{(4r+2)n^2-(6r+1)n+2r-2}~\text{by \cite[p. 36, eq. (3.3.8)]{9}}\\
	&=\frac{q^{(2n+1)^2+(2r-1)\binom{2n-1}{2}}}{(q;q^2)_{2n-1}},
\end{align*}
as desired. The other three cases, as noted previously, are perfectly analogous to this case.
\end{proof}
\begin{corollary}\label{corollary25}
For $r\ge-1$, the generating function for partitions with $n$ copies of $n$ in which the weighted difference between parts is at least $r$ is given by
\[\sum_{m\ge0}\frac{q^{m^2+4\binom{m}{2}}}{(q;q)_m(q;q^2)_m}.\]
\end{corollary}
\begin{proof}
By Lemma \ref{lemma22} and Theorem \ref{theorem20}, we see that the generating function for all partitions with $n$ copies of $n$ and having exactly $m$ parts is given by
\[\frac{\beta_r(m,q)}{(q;q)_m}=\frac{q^{m^2+r\binom{m}{2}}}{(q;q)_m(q;q^2)_m},\]
and summing over all $m$, we obtain the result.
\end{proof}
Proof of Theorem \ref{theorem18}:
\begin{proof}
This follows directly from Corollary \ref{corollary25} with $r=1$ and the identity \eqref{6.2}.
\end{proof}
Proof of Theorem \ref{theorem19}:
\begin{proof}
This follows directly from Corollary \ref{corollary25} with $r=0$ and the identity \eqref{6.3}.
\end{proof}
Proof of Theorem \ref{theorem20}:
\begin{proof}
This follows directly from Corollary \ref{corollary25} with $r=-1$ and the identity \eqref{6.4}.
\end{proof}
In addition, we can now interpret a couple of Ramanujan’s mock theta functions with partitions with $n$ copies of $n$.
\begin{theorem}\label{theorem26}
The tenth order mock theta function \cite[p. 149, eq. (8.1.2)]{14}
\[\psi_{10}(q):=\sum_{n=0}^\infty\frac{q^{\binom{n+1}{2}}}{(q;q^2)_n},\]
is the generating function for partitions with $n$ copies of $n$ where the weighted difference between parts is $-1$, and the smallest part is of the form $j_j$
\end{theorem}
\begin{proof}
We note by Theorem \ref{theorem19} that
\[\beta_{-1}(m,q)=\frac{q^{\binom{m+1}{2}}}{(q;q^2)_m}\]
and $\beta_{-1}(m,q)$ is the generating function for partitions with $n$ copies of $n$ where the weighted difference between parts is $-1$ and the smallest part is of the form $j_j$. Summing over all $m$, we obtain the result.
\end{proof}
\begin{theorem}\label{theorem27}
The third order mock theta function \cite[p. 5, eq. (2.1.3)]{14}
\[\psi_3(q):=\sum_{n=0}^\infty\frac{q^{n^2}}{(q;q^2)_n},\]
is the generating function for partitions with $n$ copies of $n$ where the weighted difference between parts is 0 and the smallest part is of the form $j_j$.
\end{theorem}
\begin{proof}
The argument here is exactly that of the proof of Theorem \ref{theorem26} with the only change being that
\[\beta_0(m;q)=\frac{q^{m^2}}{(q;q^2)_m}.\]
\end{proof}
\begin{corollary}\label{corollary28}
Let $M_1(m)$ denote the number of ordinary partitions of $m$ in which the largest part is unique and every other part occurs exactly twice. Let $M_2(m)$ denote the number of partitions of $m$ with $N$ copies of $N$ where the weighted difference between successive parts is 0 and the smallest part is of the form $j_j$. Then
\[M_1(m)=M_2(m).\]
\end{corollary}
\begin{remark}
We shall show that $\psi_3(q)$ is the generating function for both $M_1(m)$ and $M_2(m)$. As an example, consider $m=9$. $M_1(9)=4$, the partitions in question being $9,7+1+1,5+2+2,3+2+2+1+1.~M_2(9)=4$, the partitions in question being $9_9,8_6+1_1,7_3+2_2,5_1+3_1+1_1$.
\end{remark}
\begin{proof}
N. Fine \cite[p 57]{15} has observed that $\psi_3(q)$ is the generating function for partitions into odd parts without gaps. The conjugates of these partitions are the partitions enumerated by $M_1(m)$.

The result now follows from Corollary \ref{corollary28} which shows that $\psi_3(q)$ is also the generating function for $M_2(m)$.
\end{proof}
\section{Overpartitions}\label{section7}
Overpartitions were introduced by Corteel and Lovejoy \cite{14} as the natural combinatorial object counted by the coefficients of
\[\prod_{n=1}^\infty\frac{1+q^n}{1-q^n}\]
Namely these are the partitions of $n$ wherein each part size can have (or not) one summand overlined. Thus the 8 overpartitions of 3 are $3,\bar3,2+1,\bar2+1,2+\bar1,\bar2+\bar1,1+1+1$, and $1+1+\bar1$.

Among the most appealing theorems on overpartitions is Lovejoy’s extension of Schur’s theorem to overpartitions.
\begin{theorem}\label{theorem29}
The number of overpartitions of $n$ in which no part is divisible by 3 equals the number of overpartitions of $n$ wherein adjacent parts differ by at least 3 if the smaller is overlined or divisible by 3 and by at least 6 if the smaller is overlined and divisible by 3.
\end{theorem}
Now the generating function for overpartitions in which no part is divisible by 3 is:
\[\prod_{\substack{n=1\\3\not\mid n}}^\infty\frac{1+q^n}{1-q^n},\]
and this function appears in the sixth identity in L. J. Slater’s list \cite[p. 152, eq. (6) corrected]{23}.
\begin{equation}\label{7.1}
\prod_{\substack{n=1\\3\not\mid n}}^\infty\frac{1+q^n}{1-q^n}=\sum_{n\ge0}\frac{(-1;q)_nq^{n^2}}{(q;q)_n(q;q^2)_n}.
\end{equation}
It is completely unclear how exactly the right-hand side of \eqref{7.1} fits in with Lovejoy’s theorem. It turns out that \eqref{7.1} naturally fits into overpartitions with $n$ copies of $n$. Here the same principle as before applies where now only one instance of $n_j~(1\le j\le n)$ may be overlined in any partition. The generating function is:
\[\prod_{n=1}^\infty\frac{(1+q^n)^n}{(1-q^n)^n}=1+2q+6q^2+16q^3+38q^4+\cdots\]
Thus the 16 overpartitions of 3 using $n$ copies of $n$ are $3_3,\bar3_3, 3_2, \bar3_2, 3_1,\bar3_1,2_2+1_1,\bar2_2+1_1,2_2+\bar1_1,\bar2_2+\bar1_1,1_1+1_1+1_1+1_1,1_1+1_1+1_1+\bar1_1$.
\begin{theorem}\label{theorem30}
Let $J(m)$ denote the number of overpartitions of $m$ whose parts are not divisible by 3. Let $L(m)$ denote the number of overpartitions with $n$ copies of $n$ in which (i) the weighted differences between adjacent parts is $\ge0$ (ii). If the weighted difference of two or more successive parts is zero, then only the smallest part in the sequence (ignoring the subscript) can be overlined. Then for $m\ge1$,
\[J(m)=L(m).\]
\end{theorem}
As an example, when $m=4,J(4)=10$, and the overpartitions in question are $4,\bar4,2+2,2+\bar2,2+1,2+\bar1,\bar2+1,\bar2+\bar1,1+1+1+1,1+1+1+\bar1.~L(4)=10$, and the overpartitions with $n$ copies of $n$ are $4_4,\bar4_4,4_3,\bar4_3,4_2,\bar4_2,4_1,\bar4_1,3_1+1_1,3_1+\bar1_1$.
\begin{proof}[Proof of Theorem \ref{theorem30}]
This result relies heavily on the discoveries chronicled in section \ref{section6}. First let us consider
\begin{equation}\label{7.2}
\frac{(-1;q)_m}{(q;q)_m}=\frac{(1+q)(1+q^2)\cdots(1+q^{m-1})(q^m+q^n)}{(1-q)(1-q^2)\cdots(1-q^{n-1})(1-q^n)}+\frac{(1+1)(1+q)(1+q^2)\cdots(1+q^{m-1})}{(1-q)(1-q^2)\cdots(1-q^{m-1})}.
\end{equation}
The first term on the right in \eqref{7.2} generates overpartitions with exactly $m$ positive parts. The second term generates overpartitions with exactly $m$ nonnegative parts (including exactly one zero).

This dissection of \eqref{7.2} then leads directly to the desired conclusion. Namely
\[\sum_{m\ge0}\frac{(-1;q)_mq^{m^2}}{(q;q)_m(q;q^2)_m}=1+\sum_{m\ge1}\left(\frac{(-q;q)_{m-1}(q^m+q^m)}{(q;q)_m}+\frac{(-1;q)_m}{(q;q)_{m-1}}\right)\frac{q^{m^2}}{(q;q^2)_m}.\]
Now we recall from Theorem \ref{theorem20} with $r=0$ that
\[\frac{q^{m^2}}{(q;q^2)_m}\]
is the generating function for partitions with $n$ copies of $n$ having $m$ parts with smallest part of the form $j_j$.

Now instead of attaching an ordinary partition to this basic partition (as is done in Corollary \ref{corollary25}), we attach overpartitions as generated in \eqref{7.2}.

First note that these are overpartitions that are being attached. Consequently this means that if there are several identical parts being attached the result will be a sequence of parts with successive differences still 0 and with only the smallest part in the chain possibly being overlined.

Second, the first term (as noted after \eqref{7.2}) produces those partitions where the smallest summand is not of the form $j_j$, and the second term accounts for those partitions where the smallest summand is of the form $j_j$.
\end{proof}
\section{Partitions with $n$ copies of $n$ and even subscripts}\label{section8}
It may at first appear rather artificial to restrict ourselves to only those $n_j$ with $j$ even $(1<j\le n)$. However, this restriction leads to a new interpretation of one of the more striking results in L. J. Slater’s compendium \cite[p. 161, eq. (86)]{23}
\begin{equation}\label{8.1}
\sum_{n=0}^\infty\frac{q^{2n^2}}{(q;q)_{2n}}=\prod_{\substack{n=1\\n\equiv\pm2,\pm3,\pm4,\pm5\pmod{16}}}^\infty\frac{1}{1-q^n}.
\end{equation}
The right hand side of \eqref{8.1} is clearly the generating function for $G(n)$, the number of partitions of $n$ into parts $\equiv\pm2,\pm3,\pm4,\pm5\pmod{16}$. On the other hand, noting that
\[2n^2=1+1+3+3+\cdots+(2n-1)+(2n-1),\]
we may use the argument used to produce \eqref{1.3} to see that the left hand side produces partitions
\[b_1+b_2+b_3+b_4+\cdots+b_{2j-1}+b_{2j}\]
Into nondecreasing parts with $b_3-b_2\ge2,b_3-b_2\ge2,b_5-b_4\ge2,b_7-b_6\ge2,\dots$ (c.f. \cite{19}).

The reason that we renew our study of \eqref{8.1} is that it fits perfectly into the theme of the last two sections.
\begin{theorem}\label{theorem31}
Let $H(n)$ denote the number of partitions of $n$ using $n$ copies of $n$ but (i) restricted to even subscripts, (ii) the weighted differences between successive parts is $\ge0$, and (iii) excluding adjacent pairs $n_i,m_j$ with $n$ and $m$ both odd and $((n_i-m_j))=0$. Then for $n\ge1$,
\[G(n)=H(n).\]
\end{theorem}
As an example, $G(10)=7$ where the relevant partitions are $5+5,5+3+2,4+4+2,4+3+3,4+2+2+2,3+3+2+2,2+2+2+2+2.$ $H(10)=7$ where the relevant partitions are $10_{10},10_8,10_6,10_4,10_2,8_2+2_2,8_4+2_2$. Note that $7_2+3_2$ is disallowed because $((7_2-3_2))=0$ with both 7 and 3 odd.
\begin{proof}[Proof of Theorem \ref{theorem31}]
We rewrite the left side of \eqref{8.1} as
\[\sum_{m\ge0}\frac{q^{2m^2}}{(q;q)_{2m}}=\sum_{m\ge0}\frac{(-q;q^2)_m}{(q^2;q^2)_m}\frac{q^{2m^2}}{(q^2;q^4)_m}.\]
Now by conjugation of the 2 modular representations of partitions without repeated odd parts, we see that
\begin{equation}\label{8.2}
\frac{(-q;q^2)_m}{(q^2;q^2)_m}
\end{equation}
is the generating function for partitions with exactly $n$ nonnegative parts with no repeated odd parts. On the other hand,
\begin{equation}\label{8.3}
\frac{q^{2m^2}}{(q^2;q^4)_m}
\end{equation}
is merely the dilation $(q\to q^2)$ of $\beta_0(m,q)$ from Theorem \ref{theorem20}. Thus $\beta_0(m,q^2)$ is the generating function for partitions into $n$ copies of $N$ where summands are of the form $(2r)_{2s}~(1\le s\le r)$.

Now we proceed here exactly as before in Lemma \ref{lemma24} and Corollary \ref{corollary25}. The attachment of the ordinary partitions generated by \eqref{8.2} to the partitions with $n$ copies of $n$ as generated by \eqref{8.3} yields partitions into $n$ copies of $n$ with $m$ parts subject to the requirement that the weighted difference between parts is nonnegative, and that two successive parts $r_i$ and $s_j$ cannot have $((r_i-s_j))=0$ with both $r$ and $s$ odd.
\end{proof}
\section{Conclusion}\label{conclusion}
There are several points to be made in summary.

First, we have chosen a sampling of possible applications of this method to make clear its widespread utility. Thus there are many instances of Theorem \ref{theorem1} that have yet to be considered. We have treated only a few of these.

Second, there are other examples of SIP classes. Indeed, the inspiration for this paper arose from \cite{12}. The SIP class in \cite{12} is the set of integer partitions in which the parts are distinct, the smallest is even, and there are no consecutive odd parts. This SIP class is \underline{not} an instance of Theorem \ref{theorem1}. Consequently, it is surely valuable to explore SIPs not included in Theorem \ref{theorem1}.

Finally, there are other theorems that cry out for an analogous theory. For example, the mod 7 instance of the generalization of the Rogers-Ramanujan identities given in \cite{6} may be stated:
\begin{equation}\label{9.1}
	\sum_{N\ge0}\frac{q^{N^2}}{(q;q)_N}\sum_{m=0}^N\genfrac[]{0pt}{0}{N}{m}_1q^{m^2} =\prod_{\substack{n=1\\n\not\equiv0,\pm3\pmod7}}^\infty\frac{1}{1-q^n}.
\end{equation}
There are several interpretations of the left hand side of \eqref{6.1}~(\cite{8}, \cite{10}, \cite{18}); however none seems to lend itself to an SIP-style interpretation.
If such an interpretation could be found, this would open many further possibilities.
\nocite{*}

\bibliographystyle{abbrv}
	\bibliography{reference}

\noindent The Pennsylvania State University\\
University Park, PA 16802\\
gea1@psu.edu
\vspace{1em}\\

\end{document}